\documentclass[runningheads]{llncs}
\usepackage{graphicx}

\usepackage{amsmath}
\usepackage{amssymb}
\usepackage{comment}
\usepackage{cite}
\usepackage{enumerate}
\usepackage{mathrsfs}
\usepackage{mathtools}
\usepackage{proof}
\usepackage{qtree}
\usepackage{subcaption}
\usepackage{siunitx}
\usepackage{url}
\usepackage{xcolor}
\usepackage{pifont}
\definecolor{1f1e33}{HTML}{1F1E33}
\definecolor{mediumblue}{HTML}{0000CD}

\usepackage{tikz}
\usetikzlibrary{automata, arrows.meta, positioning}

\usepackage[all]{xy}

\usepackage{listings}
\usepackage[colorlinks=true, hidelinks]{hyperref}
\usepackage{nameref}
\usepackage[all]{xy}

\newcommand{\Prop}{\mathbf{Prop}}
\newcommand{\Nom}{\mathbf{Nom}}

\newcommand{\dia}{\diamondsuit}
\newcommand{\Kmodel}{\mathcal{M}}
\newcommand{\Kframe}{\mathcal{F}}
\newcommand{\tableau}{\mathbf{TAB}}

\newcommand{\hybridK}{\mathbf{K(@)}}

\newcommand{\modalIB}{\mathbf{IB}}
\newcommand{\hybridIB}{\mathbf{IB(@)}}
\newcommand{\Drest}{(\mathcal{D})}
\newcommand{\Irest}{(\mathcal{I})}

\newcommand{\dom}{\mathrm{dom}}
\newcommand{\reflexivity}{\mathit{Ref}}
\newcommand{\Id}{\mathit{Id}}
\newcommand{\squaresym}{\square_{\mathit{sym}}}
\newcommand{\tableauIB}{\mathbf{TAB}_{\mathbf{IB}}}

\date{}

\begin{document}
\title{Complete and Terminating Tableau Calculus for Undirected Graph}
%
%

\author{Yuki Nishimura\inst{1}\orcidID{0009-0000-0475-1168} \and
Tsubasa Takagi\inst{2}\orcidID{0000-0001-9890-1015}
}

\institute{Tokyo Institute of Technology, Tokyo, Japan \\ \email{nishimura.y.as@m.titech.ac.jp} \and
Japan Advanced Institute of Science and Technology, Ishikawa, Japan
\email{tsubasa@jaist.ac.jp}
}
\authorrunning{Y. Nishimura \and T. Takagi}
%
%
\maketitle              
\begin{abstract}
Hybrid logic is a modal logic with additional operators specifying nominals and is highly expressive. For example, there is no formula corresponding to the irreflexivity of Kripke frames in basic modal logic, but there is in hybrid logic. Irreflexivity is significant in that irreflexive and symmetric Kripke frames can be regarded as undirected graphs reviewed from a graph theoretic point of view. Thus, the study of the hybrid logic with axioms corresponding to irreflexivity and symmetry can help to elucidate the logical properties of undirected graphs. In this paper, we formulate the tableau method of the hybrid logic for undirected graphs. Our main result is to show the completeness theorem and the termination property of the tableau method, which leads us to prove the decidability.

\end{abstract}
\section{Introduction}

\subsection*{Background}

Hybrid logic is an extension of basic modal logic with additional propositional symbols called nominals and an operator called a satisfaction operator $@$. A nominal is a formula that is true only in a single possible world in a Kripke model. Nominals are regarded as syntactic names of possible worlds. Hybrid logic is highly expressive because nominals allow arbitrary possible worlds to be designated. Hybrid logic was first invented by Prior \cite{prior1967, prior1968} and has developed in various ways since then. (See \cite{tencate2005, blackburn2006P, areces2007_14, indrzejczak2007, brauner2011}.)

One of the advantages of hybrid logic compared to basic modal logic is that we can treat the irreflexivity of Kripke frames. In basic modal logic, no axiom corresponds to the irreflexivity of the model \cite{goldblatt1987}. In hybrid logic, however, we have a simple axiom, $@_i \neg \dia i$, where $i$ is a nominal \cite{blackburn1999}.

This advantage makes hybrid logic well-suited for dealing with undirected graphs, namely relational structures with irreflexivity and symmetry. There are some logics whose semantics are given by undirected graphs; for example, orthoframes used in (the minimal) quantum logic \cite{goldblatt1974} and social networks used in Facebook logic \cite{seligman2011, seligman2013} can be regarded as undirected graphs.

\subsection*{Motivation}

The motivation of this paper is to find a procedure that can automatically determine the validity of formulas for undirected graphs. In general, tableau calculus gives such a procedure if its completeness and the termination property are shown. That is, the decidability of the validity of formulas follows from the existence of a complete and terminating tableau calculus.

There are some previous research on tableau calculi for hybrid logic. Tableau calculi for hybrid logic was first proposed by \cite{tzakova1999} as a prefixed tableau calculus, which was based on the method of \cite{fitting1983}. Later,  \cite{bolander2007} proposed tableau calculi for hybrid logic in another way, as a labeled sequent calculus. It worked better with hybrid logic than the previous method because we can \emph{internalize} labels owing to the satisfaction operator, and we no longer need labels to construct a proof system. A subsequent paper \cite{bolander2009} of \cite{bolander2007} proposed a tableau calculus for hybrid logic based on $\mathbf{K4}$, and claimed that the loop-checking approach can be extended in it to the hybrid logics based on $\mathbf{S4}$ and $\mathbf{S5}$.

Our study, however, focuses on a combination of axioms corresponding to irreflexivity and symmetry, namely (irr) and (sym) in \cite{bolander2009}. This combination is worth considering from the graph-theoretic point of view because it characterizes undirected graphs as Kripke frames. On the other hand, Bolander and Blackburn \cite{bolander2009} have not dealt with the combination of (irr) and (sym). The reason is that they considered it challenging: in \cite{bolander2009}, they wrote, ``But then we are faced with the task of combining such conditions as (irr), (sym), (asym), (antisym), (intrans), (uniq) and (tree) with (trans), and here matters are likely to be much trickier.''

\subsection*{Results}

With the above motivation, we propose a complete and terminating tableau calculus with respect to the class of irreflexive and symmetric frames: its completeness is shown in Theorem \ref{thmcompleteIB}, and its termination property is shown in Theorem \ref{cordecidable}. For this purpose, we show key lemmas called the model existence lemma (Lemma \ref{lemmodelexistIB}) and the bulldozing lemma (Lemma \ref{lembulldozepreserve}).

To address a hybrid logic with the axioms (irr) and (sym), we employ the bulldozing method, which was originally proposed in \cite{segerberg1971}. (See also \cite{hughes1996}.) Blackburn \cite{blackburn1990} used the bulldozing method in proving the completeness of hybrid logic with the axiom (irr). Our work shows that the bulldozing method also works in dealing with the tableau calculus for hybrid logic with the rule of irreflexivity $\Irest$ in Section \ref{secsyntaxIB}.

Our work can be regarded as an extension of \cite{blackburn1990} in the following sense. The paper \cite{blackburn1990} showed that none of the reflexive points in the model created by a maximal consistent set has a nominal and checked that there is no matter to apply bulldozing. In our work, however, all the models created by an open saturated branch are named --- each world makes at least one nominal true. We show that this situation causes no problem since all the nominals true in reflexive worlds are not an element that determines the validity of a formula but just ``labels'' to point to worlds.

\subsection*{Related Work}

Examples of previous study on graph theory using hybrid logic can be found in \cite{benevides2009, benevides2011, gate2013}. They studied hybrid graph logic, a hybrid logic with a special operator $\dia^+$, as a tool for analyzing some properties of graph theory. In particular, for undirected graphs, \cite{takeuti2020} discussed how to express the planarity of undirected graphs. (Note that \cite{takeuti2020} also showed the decidability of hybrid logic for undirected graphs.)

All of these studies adopted Hilbert-style axiomatization as their proof system. On the other hand, we adopt tableau calculus as a proof system. As is well known, tableau calculi have a clear advantage compared to Hilbert-style axiomatizations in that the proof construction can be automated.

\subsection*{Organization of the Paper}

The rest of this paper is organized as follows. Section \ref{secsemantics} introduces the semantics of hybrid logic, and Section \ref{secsyntaxK} shows the tableau system of basic hybrid logic. After that, we introduce the tableau calculus corresponding to undirected graphs in Section \ref{subsecconstruct}, and then we show the termination property and the completeness theorem in Section \ref{subsecterm} and \ref{subseccomp}, respectively. Finally, in Section \ref{secconclusion}, we summarize the main result of the paper and propose future work.

\section{Preliminaries}
\subsection{Kripke Semantics of Hybrid Logic}
\label{secsemantics}

Here, we review a hybrid logic with an operator $@$. (See \cite{brauner2011} for more details.)

\begin{definition}[language]
  We have a countably infinite set $\Prop$ of \emph{propositional variables} and another countably infinite set $\Nom$ of \emph{nominals}, which is disjoint from $\Prop$. A \emph{formula} $\varphi$ of hybrid logic is defined inductively as follows:
  \begin{align*}
    \varphi \Coloneqq p \mid i \mid \neg \varphi \mid \varphi \land \varphi \mid \diamondsuit \varphi \mid @_i \varphi,
  \end{align*}
  where $p \in \Prop$ and $i \in \Nom$.

  We use the following abbreviations:
  \begin{align*}
    \varphi \lor \psi \coloneqq \neg(\neg \varphi \land \neg \psi), \ \varphi \rightarrow \psi \coloneqq \neg(\varphi \land \neg \psi), \ \square \varphi \coloneqq \neg \diamondsuit \neg \varphi.
  \end{align*}
\end{definition}

\begin{definition}
  A \emph{Kripke model} (we call it \emph{model}, in short) $\Kmodel = (W, R, V)$ is defined as follows:
  \begin{itemize}
    \item $W$ is a non-empty set,
    \item $R$ is a binary relation on $W$,
    \item $V$ is a function $V:\Prop \ \cup \ \Nom \to \mathcal{P}(W)$ such that $V(i) = \{ w \} \ \text{for some} \ w \in W$ for each $i \in \Nom$, where $\mathcal{P}(W)$ denotes the powerset of $W$.
  \end{itemize}

  Furthermore, we call $\Kframe = (W, R)$ that satisfies the conditions above a \emph{Kripke frame} (or shortly \emph{frame}).
\end{definition}

This definition reflects the property of nominals: one nominal is true in only one world. In this paper, we write $w R v$ to mean $(w, v) \in R$. Moreover, We abbreviate $w \in W$ such that $V(i) = \{ w \}$ by $i^V$.

\begin{definition}
  Given a model $\Kmodel$, a possible world $w$ in $\Kmodel$, and a formula $\varphi$, the \emph{satisfaction relation} $\Kmodel, w \models \varphi$ is defined inductively as follows:
  \begin{align*}
    \Kmodel, w \models p &\iff w \in V(p), \text{where} \ p \in \Prop; \\
    \Kmodel, w \models i &\iff w = i^V, \text{where} \ i \in \Nom; \\
    \Kmodel, w \models \neg \varphi &\iff \text{not} \ \Kmodel, w \models \varphi; \\
    \Kmodel, w \models \varphi \land \psi &\iff \Kmodel, w \models \varphi \ \text{and} \ \Kmodel, w \models \psi; \\
    \Kmodel, w \models \diamondsuit \varphi &\iff \text{there exists} \ v \ \text{such that} \ w R v \ \text{and} \ \Kmodel, v \models \varphi; \\
    \Kmodel, w \models @_i \varphi &\iff \Kmodel, i^V \models \varphi.
  \end{align*}
  \label{defsatis}
\end{definition}

\begin{definition}
    A formula $\varphi$ is said to be \emph{valid} (denoted as $\models \varphi$) if $\Kmodel, w \models \varphi$ holds for all models $\Kmodel$ and all of its worlds $w$.
    \label{defvalid}
\end{definition}

\subsection{Tableau Calculus for $\hybridK$}
\label{secsyntaxK}

$\hybridK$ is an axiomatization of hybrid logic with an operator $@$, based on the minimal normal modal logic $\mathbf{K}$. (See \cite{blackburn2006P}, for example.) We write $\tableau$ to indicate the tableau calculus for hybrid logic $\hybridK$. This section is based on \cite{bolander2007}.

\begin{definition}
  A \emph{tableau} is a well-founded tree constructed in the following way:
  \begin{itemize}
    \item Start with a formula of the form $@_i \varphi$ (called the \emph{root formula}), where $\varphi$ is a formula of hybrid logic and $i \in \Nom$ does not occur in $\varphi$.
    \item For each branch, extend it by applying rules (see Definition \ref{defrules}) to all nodes as often as possible. However, we can no longer add any formula in a branch if at least one of the following conditions is satisfied:
    \begin{enumerate}[i)]
      \item Every new formula generated by applying any rule already exists in the branch.
      \item The branch is closed. (See Definition \ref{defclose}.)
    \end{enumerate}
  \end{itemize}
  Here, a \emph{branch} means a maximal path of a tableau. If a formula $\varphi$ occurs in a branch $\Theta$, we write $\varphi \in \Theta$.
  \label{deftableau}
\end{definition}

\begin{definition}
    A branch of a tableau $\Theta$ is \emph{closed} if there exists a formula $\varphi$ and a nominal $i$ such that $@_i \varphi, @_i \neg \varphi \in \Theta$. We say that $\Theta$ is \emph{open} if it is not closed. A tableau is called \emph{closed} if all branches in the tableau are closed.
    \label{defclose}
\end{definition}

\begin{definition}
  We provide the rules of $\tableau$ in Figure \ref{figrules}.
  \begin{figure}[t]
  \begin{gather*}
\infer[{[\neg]}]
    {@_j j}
    {@_i \neg j}
\qquad
\infer[{[\neg \neg]}]
    {@_i \varphi}
    {@_i \neg \neg \varphi}
\qquad
\infer[{[\land]}]
    {\deduce{@_i \psi} {@_i \varphi}}
    {@_i (\varphi \land \psi)}
\qquad
\infer[{[\neg \land]}]
    {@_i \neg \varphi \mid @_i \neg \psi}
    {@_i \neg (\varphi \land \psi)} \\
\\
\infer[{[\diamondsuit]}^{*1, *2, *3}]
    {\deduce{@_j \varphi}{@_i \diamondsuit j}}
    {@_i \diamondsuit \varphi}
\qquad
\infer[{[\neg \diamondsuit]}]
    {@_j \neg \varphi}
    {\deduce{@_i \diamondsuit j}{@_i \neg \diamondsuit \varphi}}
\qquad
\infer[{[@]}]
    {@_j \varphi}
    {@_i @_j \varphi}
\qquad
\infer[{[\neg @]}]
    {@_j \neg \varphi}
    {@_i \neg @_j \varphi}
\qquad
\infer[{[\Id]}^{*3}]
    {@_j \varphi}
    {\deduce{@_i j} {@_i \varphi}}
  \end{gather*}

  *1: $j \in \Nom$ does not occur in the branch.

  *2: This rule can be applied only one time per formula.

  *3: The formula above the line is not an accessibility formula. Here, an \emph{accessibility formula} is the formula of the form $@_i \dia j$ generated by [$\dia$], where $j$ is a new nominal. \vspace{\baselineskip} \\
  In these rules, the formulas above the line show the formulas that have already occurred in the branch, and the formulas below the line show the formulas that will be added to the branch. The vertical line in the $[\neg \land]$ means that the branch splits to the left and right.
  \caption{The rules of $\tableau$}
  \label{figrules}
  \end{figure}
  \label{defrules}
\end{definition}

\begin{definition}[probability]
    Given a formula $\varphi$, we say that $\varphi$ is \emph{provable} in $\tableau$ if there is a closed tableau whose root formula is $@_i \neg \varphi$, where $i \in \Nom$ does not occur in $\varphi$.
\end{definition}

Bolandr and Blackburn \cite{bolander2007} showed two significant properties of $\tableau$. One is the termination property, and the other is completeness.

\begin{theorem}[termination]
  The tableau calculus $\tableau$ has the termination property. That is, for every tableau in $\tableau$, all branches of it are finite.
  \label{thmterm}
\end{theorem}

\begin{theorem}[completeness]
  The tableau calculus $\tableau$ is complete for the class of all frames.
  \label{thmcomplete}
\end{theorem}

\section{Tableau Calculus for $\hybridIB$}
\label{secsyntaxIB}

In this paper, we discuss the hybrid logic corresponding to irreflexive and symmetric models. We call the models \emph{orthomodels}, following \cite{goldblatt1974}.

\begin{definition}
  An \emph{orthoframe} is a Kripke frame $\Kframe = (W, R)$, where $R$ is irreflexive and symmetric. An \emph{orthomodel} is $\Kmodel = (\Kframe, V)$, where $\Kframe$ is an orthoframe.
\end{definition}

$\hybridIB$ is a hybrid logic corresponding to the class of all orthoframes. We can construct it by adding the following two axioms to the hybrid logic $\hybridK$:
\begin{align*}
  (irr): \ &@_i \neg \dia i, 
  (sym): \ &@_i \square \dia i.
\end{align*}

\subsection{Constructing Tableau Calculus}
\label{subsecconstruct}

Our aim is to construct a tableau calculus corresponding to $\hybridIB$, keeping the termination property. First, we add the following rule that reflects the symmetry:

\[
  \infer[{[\squaresym]}]
    {@_j \varphi}
    {\deduce{@_j \dia i}{@_i \square \varphi}} .
\]

However, this rule allows us to build an infinite branch.

\begin{example}
    Consider an example of Figure \ref{tabsym}. (Formulas with ${}^*$ are accessibility formulas.) In this branch, nominals are generated infinitely as $i_1, i_3, i_5, \ldots$ and $i_2, i_4, i_6, \ldots$.
\end{example}

\begin{figure}[t]
  \begin{align*}
    &1.\,@_{i_0} (\dia \dia i \land @_i \square \dia \dia i) & \\
    &2.\,@_{i_0} \dia \dia i &(1) \\
    &3.\,@_{i_0} @_i \square \dia \dia i &(1) \\
    &4.\,{@_{i_0} \dia i_1}^* &(2) \\
    &5.\,@_{i_1} \dia i &(2) \\
    &6.\,@_i \square \dia \dia i &(3) \\
    &7.\,{@_{i_1} \dia i_2}^* &(5) \\
    &8.\,@_{i_2} i &(5) \\
    &9.\,@_{i_1} \dia \dia i &(5, 6, [\squaresym]) \\
    &10.\,@_i i &(8) \\
    &11.\,{@_{i_1} \dia i_3}^* &(9) \\
    &12.\,@_{i_3} \dia i &(9) \\
    &13.\,{@_{i_3} \dia i_4}^* &(12) \\
    &14.\,@_{i_4} i &(12) \\
    &15.\,@_{i_3} \dia \dia i &(6, 12, [\squaresym]) \\
    &16.\,{@_{i_3} \dia i_5}^* &(15) \\
    &17.v@_{i_5} \dia i &(15) \\
    &18.\,{@_{i_5} \dia i_6}^* &(17) \\
    &19.\,@_{i_6} i &(17) \\
    &20.\,@_{i_5} \dia \dia i &(6, 17, [\squaresym]) \\
    &\qquad \vdots
  \end{align*}
  \caption{Non-terminating tableau of $\tableau$ with only $[\squaresym]$}
  \label{tabsym}
\end{figure}

In this case, the nominals $i_1, i_3, i_5, \ldots$ play the same role. For any nominal $i_{2n + 1} \ (n \geq 0)$ occurring in $\Theta$, we have $@_{i_{2n+1}} \dia \dia i, @_{i_{2n+1}} \dia i \in \Theta$. Then we have to prohibit the creation of such redundant nominals.

The solution to this problem has already been proposed in \cite[Section 5.2]{bolander2007} and \cite[Section 7]{bolander2009}. That is to add this restriction $\Drest$ to the tableau calculus.

\begin{description}
  \item[$\Drest$] The rule $[\dia]$ can only be applied to a formula $@_i \dia \varphi$ on a branch $\Theta$ if $i$ is a quasi-urfather.
\end{description}

We will give a precise definition of \emph{quasi-urfather} later. Intuitively, the quasi-urfather is the representative of the aforementioned redundant nominals.

The next thing to do is the addition of rules that address irreflexivity. The solution, which is quite similar to that of \cite{bolander2009}, is straightforward: we add the formula of the form $@_i \neg \dia i$ to the branch for every nominal $i$.

\begin{description}
  \item[$\Irest$] For any nominal $i$ occurring in $\Theta$, we add a formula $@_i \neg \dia i$.
\end{description}

\begin{example}
  Consider an example of Figure \ref{tabirr}. The branch with the root formula $@_i \dia (i \land p)$ is closed owing to $\Irest$. Indeed, the formula  $@_i \dia (i \land p)$ never holds in any orthomodel, i.e. $\neg @_i \dia (i \land p)$ is a theorem of $\hybridIB$.
\end{example}

\begin{figure}[t]
  \begin{align*}
    &1.\,@_i \dia (i \land p) & \\
    &2.\,@_i \neg \dia i &\Irest \\
    &3.\,{@_i \dia j}^* &(1) \\
    &4.\,@_j (i \land p) &(1) \\
    &5.\,@_j \neg \dia j &\Irest \\
    &6.\,@_j \neg i &(2, 3) \\
    &7.\,@_j i &(4) \\
    &8.\,@_j p &(4) \\
    &\quad \text{\ding{55}}
  \end{align*}
  \caption{A closed tableau owing to $\Irest$}
  \label{tabirr}
\end{figure}

Finally, we add a new rule $[\reflexivity]$. In this rule, $i \in \Nom$ must have already occurred in the branch. Note that if we add this rule to our tableau calculus, we no longer need $[\neg]$.
\[
    \infer[{[\reflexivity]}]
        {@_i i}
        { }
\]

Now we are ready to construct the tableau calculus for $\hybridIB$.

\begin{definition}
  $\tableau_{\hybridIB}$ (or simply $\tableau_{\modalIB}$) is the tableau calculus made by adding $[\squaresym], \Drest, \text{ and } \Irest$ to $\tableau$ and replacing $[\neg]$ with $[\reflexivity]$. Moreover, if there is a closed tableau in $\tableauIB$ whose root formula is $@_i \neg \varphi$, where $i \in \Nom$ does not occur in $\varphi$, we say that $\varphi$ is \emph{provable} in $\tableau_{\modalIB}$.
  \label{deftableauIB}
\end{definition}

\subsection{Termination}
\label{subsecterm}

A tableau calculus has the \emph{termination property} if, for every tableau, all branches in it are finite.

The first thing we do is to insert a relation between nominals occurring in a branch of a tableau. It is the case that some nominals that exist in a branch play similar roles. We need to make them be grouped together as \emph{twins}.

First, we show the quasi-subformula property. $\tableauIB$ does not have the subformula property that tableau calculus of propositional logic has. However, by extending the concept of subformula, we can prove a similar property.

\begin{definition}[quasi-subformula]
  Given two formulas of the form $@_i \varphi$ and $@_j \psi$, $@_i \varphi$ is a \emph{quasi-subformula} of $@_j \psi$ if one of the following conditions holds:
  \begin{itemize}
    \item $\varphi$ is a subformula of $\psi$.
    \item $\varphi$ has the form $\neg \chi$, and $\chi$ is a subformula of $\psi$.
  \end{itemize}
\end{definition}

\begin{lemma}
  Let $\Theta$ be a branch of a tableau. For every formula of the form $@_i \varphi$ in $\Theta$, one of the following conditions holds.
  \begin{itemize}
      \item It is a quasi-subformula of the root formula of $\Theta$.
      \item It is an accessibility formula.
      \item It is a quasi-subformula of $@_i \neg \dia i$ for some $i$ occurring in $\Theta$.
  \end{itemize}
  \label{lemquasi}
\end{lemma}
\begin{proof}
    By induction on the number of applied rules. We provide the proof only for the cases [$\dia$], $[\neg \dia]$, and $[\reflexivity]$. The other cases are left to the reader.
    \begin{description}
        \item[{$[\dia]$}] Since $@_i \dia j$ is an accessibility formula, it suffices to show that $@_j \varphi$ is a quasi-subformula of the root formula. $@_j \varphi$ is a quasi-subformula of $@_i \dia \varphi$, and by the induction hypothesis, $@_i \dia \varphi$ is a quasi-subformula of the root formula.
        \item[{$[\neg \dia]$}] $@_j \neg \varphi$ is a quasi-subformula of $@_i \neg \dia \varphi$. By the induction hypothesis, $@_i \neg \dia \varphi$ is either a quasi-subformula of the root formula or the formula $@_i \neg \dia i$. Therefore, $@_j \neg \varphi$ is a quasi-subformula of either the root formula or $@_i \neg \dia i$.
        \item[{$[\reflexivity]$}] The lemma straightforwardly follows, because $@_i i$ is a quasi-subformula of $@_i \neg \dia i$.
    \end{description}
\end{proof}

Next, we define the set $T^\Theta(i)$, a set of true formulas in the world pointed to by the nominal $i$.

\begin{definition}
  Let $\Theta$ be a branch of a tableau. For every nominal $i$ occurring in $\Theta$, the set $T^\Theta(i)$ is defined as follows:
  \[
    T^\Theta(i) = \{ \varphi \mid @_i \varphi \in \Theta \text{ and } @_i \varphi \text{ is a quasi-subformula of the root formula}\}.
  \]
  \label{defT}
\end{definition}

\begin{lemma}
  For all $i$ occurring in $\Theta$, $T^\Theta(i)$ in Definition \ref{defT} is finite.
\end{lemma}
\begin{proof}
  Straightforward from the construction of $T^\Theta(i)$.
\end{proof}

Using the set $T^\Theta(i)$ defined above, we can give a definition of twins.

\begin{definition}
    Let $\Theta$ be a branch of a tableau. We call nominals $i, j$ \emph{twins} in $\Theta$ if $T^\Theta(i) = T^\Theta(j)$.
\end{definition}

We need one more step to understand a quasi-urfather to define a generating relation between nominals.

\begin{definition}
  Let $\Theta$ be a tableau branch, and let $i$ and $j$ be nominals occurring in $\Theta$. If $j$ is introduced by applying $[\diamondsuit]$ to a formula $@_i \diamondsuit \varphi$, we say that $j$ is a \emph{generated nominal} from $i$ (denoted as $i \prec_\Theta j$).
  \label{defgen}
\end{definition}

Note that the only rule generating a new nominal is $[\diamondsuit]$. Since all the accessibility formulas emerge in a branch only when $[\diamondsuit]$ is applied to some formula in the branch, $i \prec_\Theta j$ is equivalent to that there is an accessibility formula $@_i \diamondsuit j$ in $\Theta$.

Based on these preparations, we can give a precise definition of quasi-urfather.

\begin{definition}[quasi-urfather]
  We call a nominal $i$ \emph{quasi-urfather} on $\Theta$ if there are no twins $j, k$ such that $j \neq k$ and $j, k \prec_\Theta^* i$, where $\prec_\Theta^*$ denotes a reflexive and transitive closure of $\prec_\Theta$.
  \label{defqur}
\end{definition}

Before proving the termination property, we see an example that $\Drest$ plays an important role in stopping extending a branch.

\begin{example}
    With $\Drest$, the tableau in Figure \ref{tabsym} stops before adding the 16th formula $@_{i_3} \dia i_5$. In this branch, up to the 15th formula $@_{i_3} \dia \dia i$, nominals $i_1$ and $i_3$ are twins, since $T_\Theta(i_1) = T_\Theta(i_3) = \{ \dia i, \dia \dia i \}$. Then $i_3$ is not a quasi-urfather, so we can no longer apply $[\dia]$ to the 15th formula $@_{i_3} \dia \dia i$.
\end{example}

Now we move on to prove the termination property.

\begin{lemma}
  Given a branch $\Theta$ of a tableau, we define a structure $G = (N^\Theta, \prec_\Theta)$ where
  \begin{itemize}
    \item $N^\Theta$ is the set of nominals occurring in $\Theta$.
    \item $\prec_\Theta$ is the binary relation defined in Definition \ref{defgen}.
  \end{itemize}
  Then $G$ is a finite set of well-founded and finitely branching trees.
  \label{lemG}
\end{lemma}
\begin{proof}
  First, we show that $G$ is well-founded. To show that, Suppose otherwise, i.e., there exists an infinite path in $G$ such that
  \[
    i_0 \succ_\Theta i_1 \succ_\Theta i_2 \succ_\Theta \cdots .
  \]
  For the restriction of applying $[\diamondsuit]$, each nominal $i_0, i_1, i_2, \dots$ is different from any others. By the definition of $\prec_\Theta$, $i_{n+1}$ occurs before $i_n$. However, since $\Theta$ is well-founded, that is a contradiction.

  The next work is to prove that all nodes in $G$ have only finite branches. However, it is directly proved from two facts: $T^\Theta(i)$ has only finite elements for all $i$, and for each $\varphi \in T^\Theta(i)$, $@_i \varphi$ generates at most one new nominal.

  Finally, we show that $G$ is a finite set of trees. The only nominals which can be a root of $G$ are the nominals in the root formula. Since the root formula has a finite length, nominals in the root formula are also finite. Moreover, it cannot be the case that for some $i, j, k \in G$ we have both $i \prec_\Theta k$ and $j \prec_\Theta k$, since $[\diamondsuit]$ requests not to use a nominal which has already occurred.
\end{proof}

\begin{lemma}
  Let $\Theta$ be a branch of a tableau. Then $\Theta$ is infinite if and only if we have the following infinite sequence:
  \[
    i_0 \prec_\Theta i_1 \prec_\Theta i_2 \prec_\Theta \cdots .
  \]
  \label{leminf}
\end{lemma}
\begin{proof}
  Since the if-part is immediate, we only show the only-if-part. Suppose $\Theta$ is infinite. Then $\Theta$ has an infinite number of nominals; if otherwise, the number of formulas occurring in $\Theta$ are finite by Lemma \ref{lemquasi}, but this contradicts that $\Theta$ is infinite. Thus, for $G = (N^\Theta, \prec_\Theta)$ in Lemma \ref{lemG}, $N^\Theta$ is an infinite set. However, since $G$ is a finite set of well-founded and finitely branching trees, there is a tree and its infinite path in $G$ by K\"{o}nig's lemma, as desired.
\end{proof}

\begin{theorem}
  The tableau calculus $\tableauIB$ has the termination property.
\end{theorem}
\begin{proof}
  By \textit{reductio ad absurdum}.

  Suppose there is an infinite branch $\Theta$. Then by Lemma \ref{leminf}, there is an infinite sequence of nominals as follows:
  \[
    i_0 \prec_\Theta i_1 \prec_\Theta i_2 \prec_\Theta \cdots .
  \]
  For $\Theta$ and its root formula $@_i \varphi$, we define $Q$ and $n$ as follows:
  \[
    Q = \{ \psi \mid @_i \psi \text{ is a quasi-subformula of } @_i \varphi \},
  \]
  and $n$ is the number of elements in $Q$. Also, let $\Theta'$ be a fragment of $\Theta$ up to, but not including, the first occurrence of $i_{2^n + 1}$. Then $i_{2^n + 1}$ is generated by applying $[\dia]$ to some $@_{i_{2^n}} \dia \varphi$. Taking $\Drest$ into consideration, $i_{2^n}$ is a quasi-urfather.

  However, since all of $T^{\Theta'}(i_0), T^{\Theta'}(i_1), \ldots, T^{\Theta'}(i_{2^n})$ are subsets of $Q$ and $n$ is the cardinality of $Q$, there exists a pair $0 \leq l, m\leq 2^n$ such that $T^{\Theta'}(i_l) = T^{\Theta'}(i_m)$ by the pigeonhole principle. Therefore, we find the twins $i_l$ and $i_m$ such that $i_l, i_m \prec_\Theta^* i_{2^n}$, but that contradicts that $i_{2^n}$ is a quasi-urfather.
\end{proof}

\subsection{Completeness}
\label{subseccomp}

Next, we prove the completeness. The basic strategy is the same as in \cite[Section 5.2]{bolander2007}. That is, define an \emph{identity urfather} and use it to create a model from an open saturated branch of tableau. When constructing the model, make sure that the relation $R^\Theta$ is symmetric.

\begin{definition}
    A branch $\Theta$ of a tableau is \emph{saturated} if every new formula generated by applying some rules already exists in $\Theta$. We call a tableau \emph{saturated} if all branches of it are saturated.
    \label{defsat}
\end{definition}

The following are some properties that a saturated branch $\Theta$ has.
\begin{itemize}
    \item If $@_i (\varphi \land \psi) \in \Theta$, then $@_i \varphi, @_i \psi \in \Theta$.
    \item If $@_i \dia \varphi \in \Theta$ which is not an accessibility formula and $i$ is a quasi-urfather, then there exists some $j \in \Nom$ such that $@_i \dia j, @_j \varphi \in \Theta$.
    \item If $@_i \square \varphi$ and $@_j \diamondsuit i \in \Theta$, then $@_j \varphi \in \Theta$.
    \item For all $i$ occurring in $\Theta$, $@_i \neg \dia i \in \Theta$.
\end{itemize}

\begin{definition}
  Let $\Theta$ be a tableau branch and $i$ a nominal occurring in $\Theta$. The \emph{identity urfather} of $i$ on $\Theta$ (denoted as $v_\Theta(i)$) is the earliest introduced nominal $j$ satisfying the following conditions:
  \begin{enumerate}
    \item $j$ is a twin of $i$.
    \item $j$ is a quasi-urfather.
  \end{enumerate}
  \label{defiur}
\end{definition}

Note that there may be no identity urfather of some $i$. If $v_\Theta(i)$ exists for given $i$, we write $i \in \dom(v_\Theta)$, which is widely used notation in mathematics.

Now we show some properties of $\dom(v_\Theta)$.

\begin{lemma}
    Let $\Theta$ be a branch of a tableau. If $i$ occurs in the root formula of $\Theta$, then $i \in \dom(v_\Theta)$.
    \label{lemrightiur}
\end{lemma}
\begin{proof}
  Suppose $i$ occurs in the root formula of $\Theta$. Then there exists no nominal $j$ such that $j \prec_\Theta i$, since when we use $[\dia]$, we must introduce a new nominal. Then $i$ is a quasi-urfather. Therefore, $v_\Theta(i)$ is definable.
\end{proof}

\begin{lemma}
    Let $\Theta$ be a saturated branch and $i$ be a quasi-urfather on $\Theta$. If $i \prec_\Theta j$, then $j \in \dom(v_\Theta)$.
    \label{lemiursucc}
\end{lemma}
\begin{proof}
    If $j$ is a quasi-urfather, the lemma immediately holds, so suppose otherwise. Then there exists a pair of twins $k, k'$ such that $k \prec_\Theta^* k' \prec_\Theta^* j$. Since $i$ is a quasi-urfather, either $k \not \prec_\Theta^* i$ or $k' \not \prec_\Theta^* i$ holds. Now assume $k' \not \prec_\Theta^* i$. Then for the facts that $k' \prec_\Theta^* j$, $i \not \prec_\Theta j$, and $G = (N^\Theta, \prec_\Theta)$ in Lemma \ref{lemG} consists of trees, we acquire $k' = j$. Thus, $j$ and $k$ are twins. Also for the facts that $k \prec_\Theta^* k'$, $i \not \prec_\Theta j = k$, and $G$ is a set of trees, we have $k \prec_\Theta^* i$. Thus $k$ is a quasi-urfather. Therefore, $v_\Theta(j)$ is definable, and it can be $k$.
\end{proof}

Next, we show some properties of $v_\Theta$.

\begin{lemma}
    Let $\Theta$ be a branch. Suppose $@_i \varphi \in \Theta$ is a quasi-subformula of the root formula of $\Theta$ and $i \in \dom(v_\Theta)$. Then $@_{v_\Theta(i)} \varphi \in \Theta$.
    \label{lemiurcl}
\end{lemma}
\begin{proof}
    If $v_\Theta(i) = i$, there is nothing to prove. Even if else, $v_\Theta(i)$ and $i$ are twins. Then the lemma holds by the definition of twins.
\end{proof}

\begin{lemma}
    Let $\Theta$ be a saturated branch. If $@_i j \in \Theta$ and $i, j \in \dom(v_\Theta)$, then $v_\Theta(i) = v_\Theta(j)$.
    \label{lemiureq}
\end{lemma}
\begin{proof}
    Suppose $@_i j \in \Theta$. Since $\Theta$ is saturated, we obtain $@_i i \in \Theta$ from $[\reflexivity]$. Then by $[\Id]$, we have $@_j i \in \Theta$. They mean that we can add all the formulas of the form $@_j \varphi$ to $\Theta$ where $@_i \varphi \in \Theta$, and vice versa. Thus, we have $ T^\Theta(i) = T^\Theta(j)$. Therefore, by definition, $v_\Theta(i) = v_\Theta(j)$.
\end{proof}

\begin{lemma}
    Let $\Theta$ be a branch. A nominal $i$ is the identity urfather if and only if $v_\Theta(i) = i$.
    \label{lemiurid}
\end{lemma}
\begin{proof}
    Straightforward from the definition of identity urfathers.
\end{proof}

\begin{definition}
    Given an open saturated branch $\Theta$ with a root formula $@_{i_0} \varphi_0$ of a tableau in $\tableau_\modalIB$, a model $\Kmodel^\Theta = (W^\Theta, R^\Theta, V^\Theta)$ is defined as follows:
    \begin{align*}
        W^\Theta &= \{ i \mid i \text{ is an identity urfather on } \Theta \}; \\
        R^\Theta &= \{ (v_\Theta(i), v_\Theta(j)) \mid @_i \dia j \in \Theta \text{ and } i, j \in \dom(v_\Theta) \} \\
        &\cup \{ (v_\Theta(j), v_\Theta(i)) \mid @_i \dia j \in \Theta \text{ and } i, j \in \dom(v_\Theta) \}; \\
        V^\Theta(p) &= \{ v_\Theta(i) \mid @_i p \in \Theta \},\text{ where } p \in \Prop; \\
        V^\Theta(i) &=
        \begin{cases}
            \{ v_\Theta(i) \} & \text{if } i \in \dom(v_\Theta), \\
            \{ i_0 \} & \text{otherwise,}
        \end{cases}
        \text{ where } i \in \Nom.
    \end{align*}
    \label{defIBtableaumodel}
\end{definition}

Then we obtain the model existence lemma, which ensures that if we have a tableau with a root formula $@_{i_0} \varphi_0$ one of whose branches are open and saturated, there exists a model which falsifies $\varphi_0$.

\begin{lemma}[model existence lemma]
    Let $\Theta$ be an open saturated branch of a tableau in $\tableau_\modalIB$ and $@_i \varphi$ be a quasi-subformula of the root formula $@_{i_0} \varphi_0$ of $\Theta$, where $i$ denotes an identity urfather. Then we have the following proposition:
    \[
        \text{if } @_i \varphi \in \Theta, \text{ then } \Kmodel^\Theta, v_\Theta(i) \models \varphi.
    \]
    In particular, $\Kmodel^\Theta, v_\Theta(i_0) \models \varphi_0$.
    \label{lemmodelexistIB}
\end{lemma}
\begin{proof}
    By induction on the complexity of $\varphi$.
    \begin{description}
        \item[{$[\varphi = p, \neg p]$}] Straightforward.
        \item[{$[\varphi = j]$}] Suppose $@_i j \in \Theta$, where $j \in \Nom$. Since $i$ is an identity urfather and $j \in \dom(v_\Theta)$ (Lemma \ref{lemrightiur}), we have $v_\Theta(i) = v_\Theta(j)$ by Lemma \ref{lemiureq}. Therefore, $\Kmodel^\Theta, v_\Theta(i) \models j$.
        \item[{$[\varphi = \neg j]$}] Suppose $@_i \neg j \in \Theta$, where $j \in \Nom$. First, by Lemma \ref{lemiurcl}, we have $@_{v_\Theta(i)} \neg j \in \Theta$. Second, since $\Theta$ is saturated, by $[\reflexivity]$, we have $@_j j \in \Theta$. Since $j$ is also a quasi-subformula of the root formula, by Lemma \ref{lemiurcl}, we have $@_{v_\Theta(j)} j \in \Theta$. Hence, $@_{v_\Theta(i)} \neg j, @_{v_\Theta(j)} j \in \Theta$. However, since $\Theta$ is open, $v_\Theta(i) \neq v_\Theta(j)$. Then by the definition of $V^\Theta$, it follows that $\Kmodel^\Theta, v_\Theta(i) \models \neg j$.
        \item[{$[\varphi = (\psi_1 \land \psi_2), \neg (\psi_1 \land \psi_2)]$}] In these cases, the proofs are simple. Then we leave them to the reader.
        \item[{$[\varphi = \dia \psi]$}] Suppose $@_i \dia \psi \in \Theta$. From the facts that $\Theta$ is saturated, $i$ is an identity urfather, and $@_i \dia \psi$ is not an accessibility formula, by $[\dia]$ there exists a nominal $j \in \Nom$ such that $@_i \dia j, @_j \psi \in \Theta$. By Lemma \ref{lemiursucc}, there exists a $v_\Theta(j)$. Hence, we have $v_\Theta(i) R^\Theta v_\Theta(j)$ and $\Kmodel^\Theta, v_\Theta(j) \models \psi$ from the definition of $R^\Theta$ and the induction hypothesis respectively. Therefore, $\Kmodel^\Theta, v_\Theta(i) \models \dia \psi$.
        \item[{$[\varphi = \neg \dia \psi]$}] Suppose $@_i \neg \dia \psi \in \Theta$, and take a nominal $j \in \Nom$ such that $v_\Theta(i) R^\Theta v_\Theta(j)$. (If nothing, $\Kmodel^\Theta, v_\Theta(i) \models \neg \dia \psi$ is straightforward.) Then by the definition of $R^\Theta$, either $@_i \dia j$ or $@_j \dia i$ is in $\Theta$. If the former exists, by $[\neg \dia]$ and the saturation of $\Theta$ we have $@_j \neg \psi$. If else, we can obtain it using $[\squaresym]$. Then in both cases, $\Kmodel^\Theta, v_\Theta(j) \models \neg \psi$ holds by the induction hypothesis. This holds for all $j$, so we have $\Kmodel^\Theta, v_\Theta(i) \models \neg \dia \psi$.
        \item[{$[\varphi = @_j \psi]$}] Suppose $@_i @_j \psi \in \Theta$. Since $\Theta$ is saturated, by $[@]$ we have $@_j \psi \in \Theta$. By the induction hypothesis, $\Kmodel^\Theta, v_\Theta(j) \models \psi$ holds. Therefore, $\Kmodel^\Theta, v_\Theta(i) \models @_j \psi$.
        \item[{$[\varphi = \neg @_j \psi]$}] We can prove similarly to the case of $\varphi = @_j \psi$.
    \end{description}
\end{proof}

To show the completeness theorem for $\tableauIB$, we have to show the symmetry and irreflexivity of $\Kmodel^\Theta$. However, it is not the case that $\Kmodel^\Theta$ is always irreflexive. The following is an example of $\Kmodel^\Theta$ that is not irreflexive. 


\begin{example}
    Consider the branch $\Theta$ in Figure \ref{tabsym}. Note that in $\tableauIB$ this branch is stopped up to the 15th formula $@_{i_3} \dia \dia i$. From $\Theta$, we create a model as Definition \ref{defIBtableaumodel}. Since $i_1$ and $i_3$ are twins and $i_2, i_4$, and $i$ are all twins, we have $v_\Theta(i_3) = i_1$ and $v_\Theta(i_2) = v_\Theta(i_4) = i$. Then $\Kmodel^\Theta$ must be as follows (Figure \ref{figrefpoint} illustrates this model):
    \begin{align*}
        W^\Theta &= \{ i_0, i_1, i \}, \\
        R^\Theta &= \{ (i_0, i_1), (i_1, i_0), (i_1, i_1), (i_1, i), (i, i_1) \}, \\
        V^\Theta(i) &= \{ i \}.
    \end{align*}
    This model has the reflexive world $i_1$.
    \begin{figure}[t]
        \centering
        \begin{tikzpicture}
            \node (w0) [state] {$i_0$};
            \node (w1) [state, right = of w0] {$i_1$};
            \node (w2) [state, right = of w1] {$i$};
            \path [->, >=stealth]
                (w0) edge (w1)
                (w1) edge (w0)
                (w1) edge [loop above] (w1)
                (w1) edge (w2)
                (w2) edge (w1);
        \end{tikzpicture}
        \caption{The model $\Kmodel_\Theta$ constructed from a branch $\Theta$ in Figure \ref{tabsym}}
        \label{figrefpoint}
    \end{figure}
\end{example}

To solve this problem, we \emph{bulldoze} all the reflexive points and turn a frame into an irreflexive one. For modal logic, the bulldozing method was originally proposed in \cite{segerberg1971}. Before giving a formal definition, we cite the explanation of bulldozing by Hughes and Cresswell \cite{hughes1996}.

\begin{quotation}
    $\ldots$ For if we take any reflexive world in any model, i.e., any world which can see itself, and replace it by a pair of worlds each able to see the other but neither able to see itself, \emph{and} we give each variable the same value in each world in the new pair as it had in the original world, $\ldots$ \cite[p. 176]{hughes1996}
\end{quotation}

For hybrid logic, Blackburn \cite{blackburn1990} showed that bulldozing works well in proving the completeness of Hilbert-style axiomatization of irreflexive hybrid logic. Our following lemmas claim that the method also works in the tableau calculus.

\begin{definition}
    Given a model $\Kmodel = (W, R, V)$, the \emph{bulldozed model} $\Kmodel_B = (W_B, R_B, V_B)$ is defined as follows:
    \begin{enumerate}[1. ]
      \item Put $W_r = \{ w \in W \mid w R w \}$.
      \item Define $W_B = W^- \cup \{ (w, n) \mid w \in W_r, n = 0, 1\}$, where $W^- = W \setminus W_r$.
      \item Define $\alpha: W_B \rightarrow W$ by $\alpha(w) = w$ if $w \in W^-$, and $\alpha((w, n)) = w$ otherwise.
      \item Define $R_B$ by $w R_B v$ if one of the following conditions holds:
      \begin{itemize}
        \item $w \in W^-$ or $v \in W^-$, and $\alpha(w) R \alpha(v)$,
        \item $w = (w', m), v = (v', n)$, $w' \neq v'$, and $w' R v'$,
        \item $w \neq v$ and $\alpha(w) = \alpha(v)$.
      \end{itemize}
      \item Define $V_B: \Prop \cup \Nom \rightarrow \mathcal{P}(W_B)$ as follows:
      \begin{itemize}
          \item If $p \in \Prop$, then $w \in V_B(p)$ if and only if $\alpha(w) \in V(p)$.
          \item If $i \in \Nom$, then $V_B(i) = \{ (w, 0) \}$, where $w = i^V$ for $w \in W_r$, and $V_B(i) = V(i)$ otherwise.
      \end{itemize}
    \end{enumerate}
  \label{defbulldoze}
\end{definition}

It can be checked that for any model, a bulldozed one is irreflexive and symmetric. The following is an example.

\begin{example}
    Let $\Kmodel$ be the model in Figure \ref{figrefpoint}. The bulldozed model $\Kmodel_B$ is illustrated in Figure \ref{figbulldoze}. Note that the proposition $i_1$ (and also $i_3$) only holds at $(i_1, 0)$ and fails in $(i_1, 1)$.
    \begin{figure}[t]
        \centering
        \begin{tikzpicture}
            \node (w0) [state] at (-2, 0) {$i_0$};
            \node (w10) [state] at (0, 1) {$(i_1, 0)$};
            \node (w11) [state] at (0, -1) {$(i_1, 1)$};
            \node (w2) [state] at (2, 0) {$i$};
            \path [->, >=stealth]
                (w0) edge (w10)
                (w10) edge (w0)
                (w0) edge (w11)
                (w11) edge (w0)
                (w10) edge (w11)
                (w11) edge (w10)
                (w10) edge (w2)
                (w2) edge (w10)
                (w11) edge (w2)
                (w2) edge (w11);
        \end{tikzpicture}
        \caption{The bulldozed model $\Kmodel_\Theta^B$ constructed from a branch $\Theta$ in Figure \ref{tabsym}}
        \label{figbulldoze}
    \end{figure}
\end{example}

Unlike in the case of modal logic, this construction cannot completely preserve satisfaction. In the model in Figure \ref{figrefpoint}, we have $\Kmodel, i_1 \models \dia i_1$, but the bulldozed model in Figure \ref{figbulldoze}, $\Kmodel_B, (i_1, 0) \models \dia i_1$ does not hold.

However, this method works well if we consider only the quasi-subformulas of the root formula. What is important is that the reflexive worlds in a model from a tableau branch have no information about nominals in the root formula.

\begin{definition}
    Let $\Theta$ be an open saturated branch and $\Kmodel^\Theta = (W^\Theta, R^\Theta, V^\Theta)$ be a model constructed in the way of Definition \ref{defIBtableaumodel}. A world $i \in W^\Theta$ is \emph{named by the root formula of $\Theta$} if there exists a nominal $j$ such that $j \in T^\Theta(i)$.
\end{definition}

\begin{lemma}
    Let $W^\Theta_r = \{ w \in W^\Theta \mid w R^\Theta w \}$. If $i \in W^\Theta_r$, then $i$ is not named by the root formula of $\Theta$.
    \label{lemnotnamed}
\end{lemma}
\begin{proof}
    By \textit{reductio ad absurdum}.

    Suppose $i$ is named by the root formula of $\Theta$. From $i \in W^\Theta_r$, we have $i R^\Theta i$, which means that there are nominals $j$ and $k$ such that $v_\Theta(j) = v_\Theta(k) = i$ and $@_j \dia k \in \Theta$. (They do not have to be different; for example, it is possible that $j = k$.) By the assumption, there exists another nominal $l$ such that $l \in T^\Theta(i)$. By the definition of identity urfather, all of $i, j, k$ are twins. Thus, we have another nominal $l$ such that $l \in T^\Theta(j)$ and $l \in T^\Theta(k)$, which means that both $@_j l$ and $@_k l$ are in $\Theta$. Therefore, we can show that $\Theta$ is closed as Figure \ref{tabnamedclose}, which is a contradiction.
    \begin{figure}[t]
        \begin{align*}
            &1.\,@_j \dia k & \\
            &2.\,@_j l & \\
            &3.\,@_k l & \\
            &4.\,@_j \neg \dia j &\Irest \\
            &5.\,@_j j &[\reflexivity] \\
            &6.\,@_k \neg j &(1, 4, [\neg \dia]) \\
            &7.\,@_l j &(2, 5, [\Id]) \\
            &8.\,@_l \neg j &(3, 6, [\Id]) \\
            &\quad \text{\ding{55}} &
        \end{align*}
        \caption{A closed branch $\Theta$ if $i \in W^\Theta_r$ is named in $\Theta$}
        \label{tabnamedclose}
    \end{figure}
\end{proof}

\begin{lemma}
    Let $\Theta$ be an open saturated branch, and $@_{i_0} \varphi_0$ be a root formula. For any nominal $i \in W_r^\Theta$ and formula $\varphi$ such that $@_i \varphi$ is a quasi-subformula of $@_{i_0} \varphi_0$, we have
    \[
        \Kmodel^\Theta_B, (i, 0) \models \varphi \iff \Kmodel^\Theta_B, (i, 1) \models \varphi .
    \]
    \label{lemwrequiv}
\end{lemma}
\begin{proof}
    By induction on the complexity of $\varphi$. Note that we do not have to consider the case $\varphi = j$ because, by Lemma \ref{lemnotnamed}, $\varphi$ cannot be any nominal $i$ occurring in $\varphi_0$.

    We give the proof only for the cases $\varphi = p, \dia \psi$, and $@_j \psi$ and only for left to right. The other cases are left to the reader.
    \begin{description}
        \item[{$[\varphi = p]$}] If $\Kmodel^\Theta_B, (i, 0) \models p$, we have $(i, 0) \in V^\Theta_B(p)$. By the construction of $V^\Theta_B$, it follows that $(i, 1) \in V^\Theta_B(p)$. Therefore, $\Kmodel^\Theta_B, (i, 1) \models p$.
        \item[{$[\varphi = \dia \psi]$}] Suppose $\Kmodel^\Theta_B, (i, 0) \models \dia \psi$. Then there exists a world $w \in W^\Theta_B$ such that $(i, 0) R^\Theta_B w$ and $\Kmodel^\Theta_B, w \models \psi$. If $w = (i, 1)$, by the induction hypothesis, we have $\Kmodel^\Theta_B, (i, 0) \models \psi$. Thus, we acquire that $\Kmodel^\Theta_B, (i, 1) \models \dia \psi$. If $w \neq (i, 1)$, we have $(i, 1) R^\Theta_B w$ by the construction of $R^\Theta_B$. Therefore, we have $\Kmodel^\Theta_B, (i, 1) \models \dia \psi$.
        \item[{$[\varphi = @_j \psi]$}] Suppose $\Kmodel^\Theta_B, (i, 0) \models @_j \psi$. Then we have $\Kmodel^\Theta_B, j^{V^\Theta_B} \models \psi$. Therefore, we acquire that $\Kmodel^\Theta_B, (i, 1) \models @_j \psi$.
    \end{description}
\end{proof}

\begin{lemma}[bulldozing lemma]
  Let $\Theta$ be an open saturated branch. Also, let $i$ be an identity urfather of $\Theta$ and $\varphi$ be a formula such that $@_i \varphi$ be a quasi-subformula of the root formula. Then we have
  \begin{center}
    $\Kmodel^\Theta, i \models \varphi \iff \Kmodel^\Theta_B, i_B \models \varphi$,
  \end{center}
  where
  \[
    i_B =
    \begin{cases}
      \{ (i, 0) \} &\text{if } i \in W_r, \\
      i & \text{otherwise.}
    \end{cases}
  \]
  \label{lembulldozepreserve}
\end{lemma}
\begin{proof}
    By induction on the complexity of $\varphi$.
    \begin{description}
        \item[{$[\varphi = p, \neg p]$}] Straightforward from the definition of $V^\Theta_B$.
        \item[{$[\varphi = j]$}] Note that $i$ is not in $W^\Theta_r$ owing to Lemma \ref{lemnotnamed}. Then we can show the equivalence in the following way.
        \begin{align*}
            \Kmodel^\Theta, i \models j &\iff V^\Theta(j) = \{ i \} \\
            &\iff V^\Theta_B(j) = \{ i_B \} \\
            &\iff \Kmodel^\Theta_B, i_B \models j .
        \end{align*}
        \item[{$[\varphi = \neg j]$}] Similar to the case $\varphi = \neg p$.
        \item[{$[\varphi = \psi_1 \land \psi_2, \neg (\psi_1 \land \psi_2)]$}] Straightforward.
        \item[{$[\varphi = \dia \psi]$}] First, we show the left-to-right direction.

        Suppose $\Kmodel^\Theta, i \models \dia \psi$. Then there exists a world $j \in W^\Theta$ such that $i R^\Theta j$ and $\Kmodel^\Theta, j \models \psi$. If $j \neq i$, we have $i_B R^\Theta j_B$ by the definition of $R^\Theta_B$, and $\Kmodel^\Theta_B, j_B \models \psi$ by the induction hypothesis. Therefore, $\Kmodel^\Theta_B, v_\Theta(i)_B \models \dia \psi$.

        The case to watch out for is $j = i$. In this case, we have $(i, 0) R^\Theta (i, 1)$, and $\Kmodel^\Theta_B, (i, 0) \models \psi$ by the induction hypothesis. However, by Lemma \ref{lemwrequiv}, we have $\Kmodel^\Theta_B, (i, 1) \models \psi$. Therefore, $\Kmodel^\Theta_B, (i, 0) \models \dia \psi$.

        The proof of the other direction is straightforward since for every $w, v \in \Theta$, $w R^\Theta_B v$ implies $\alpha(w) R^\Theta \alpha(v)$.
        \item[{$[\varphi = \neg \dia \psi]$}] Similar to the case $\varphi = \dia \psi$.
        \item[{$[\varphi = @_j \psi]$}] We only show from right to left, and the other is left to the reader.

        Suppose $\Kmodel^\Theta_B, i_B \models @_j \psi$. Take a world $w \in W^\Theta_B$ such that $V^\Theta_B(j) = \{ w \}$, and we have $\Kmodel^\Theta_B, w \models \psi$. Considering that $V^\Theta(j) = \{ v_\Theta(j) \}$, it follows that $w = (v_\Theta(j))_B$. Thus, by induction hypothesis, we have $\Kmodel^\Theta, v_\Theta(j) \models \psi$. Therefore, $\Kmodel^\Theta, i \models @_j \psi$ holds.
        \item[{$[\varphi = \neg @_j \psi]$}] Similar to the case $\varphi = @_j \psi$.
    \end{description}
\end{proof}

Finally, we reach our goal.

\begin{theorem}[completeness]
    The tableau calculus $\tableauIB$ is complete for the class of all orthoframes.
    \label{thmcompleteIB}
\end{theorem}
\begin{proof}
    We show the contraposition. Suppose $\varphi$ is not provable in $\tableauIB$. Then we can find a tableau in $\tableauIB$, whose root formula is $@_i \neg \varphi$ where $i$ does not occur in $\varphi$, and which has an open and saturated branch $\Theta$. Then by Lemma \ref{lemmodelexistIB}, we have $\Kmodel^\Theta, v_\Theta(i) \models \neg \varphi$. By Lemma \ref{lembulldozepreserve}, we have $\Kmodel^\Theta_B, v_\Theta(i)_B \models \neg \varphi$. Moreover, by the construction, $\Kmodel^\Theta_B$ is an orthomodel. Therefore, we can find an orthoframe which falsifies $\varphi$.
\end{proof}

\begin{corollary}
  $\hybridIB$ is decidable.
  \label{cordecidable}
\end{corollary}
\begin{proof}
  From the termination property of $\tableauIB$ and the fact that $\Kmodel^\Theta_B$ is finite.
\end{proof}
Note that Takeuti and Sano \cite{takeuti2020} also showed the decidability of $\hybridIB$.

\section{Conclusion and Future Work}
\label{secconclusion}
In this paper, we constructed a tableau calculus that has the termination property and is complete with respect to all frames of undirected graphs. To prove completeness, the bulldozing method worked well in dealing with tableau calculus as well as Hilbert-style axiomatization.

The first task that comes to our mind as a continuation of this paper --- but with difficulty --- is to create a tableau calculus corresponding to $\mathbf{I4(@)}$. The bulldozing method for a transitive model was proposed by \cite[Section 4.5]{blackburn2002}. Furthermore, a method to realize transitivity was also proposed in \cite{bolander2009}. Combining these methods with the method described in this paper, it should be possible to construct a tableau calculus for the class of irreflexive and transitive frames.

Moreover, there is much remaining work to formulate tableau calculi corresponding to frames with various conditions. Many axioms characterizing some frame conditions were proposed in \cite{bolander2009}, such as anti-symmetry, trichotomy, and tree-like. If there is a relational structure that is widely studied, creating tableau calculus corresponding to it is worthwhile.

Another direction of future work is to add other operators. Besides the satisfaction operator $@$, hybrid logic can have more operators, such as the existential operator $E$ and the downarrow operator $\downarrow$. Moreover, hybrid graph logic in \cite{benevides2009} contains other modal operators $\square^+$ and $\dia^+$, and \cite{takeuti2020} introduced new modal operators $\square^*$ and $\dia^*$ to express the planarity.

\section*{Acknowledgements}

We would like to thank Prof. Ryo Kashima and Prof. Katsuhiko Sano for their invaluable advice in writing this paper.
The research of the first author was supported by JST SPRING, Grant Number JPMJSP2106.
The research of the second author was supported by Grant-in-Aid for JSPS Research Fellow Grant Number JP22KJ1483.

%
%
%
\bibliographystyle{splncs04}
\bibliography{logic}

\end{document}